\documentclass [reqno,12pt]{amsart}
\usepackage{amsfonts}
\usepackage{amsmath}
\usepackage{amssymb}
\usepackage{graphicx}
\newtheorem{theorem}{Theorem}[section]
\theoremstyle{plain}

\newtheorem{corollary}{Corollary}[section]

\newtheorem{definition}{Definition}[section]
\newtheorem{example}{Example}[section]

\newtheorem{lemma}{Lemma}[section]

\newtheorem{proposition}{Proposition}[section]
\newtheorem{remark}{Remark}[section]

\numberwithin{equation}{section} \textheight  22 true cm \textwidth  15 true cm  
\begin{document}
\title[Conformal slant submersions in contact geometry]{%
CONFORMAL SLANT SUBMERSIONS IN  CONTACT GEOMETRY}
\author{Y\i lmaz G\"{u}nd\"{u}zalp}
\address{Dicle University, Faculty of  Sciences, Department of Mathematics, 21280, Diyarbak\i r, Turkey}
\email{ygunduzalp@dicle.edu.tr}
\author{Mehmet Akif Akyol}
\address{Bing\"{o}l University, Faculty of Arts and Sciences, Department of Mathematics, 12000, Bing\"{o}l,
Turkey}
\email{mehmetakifakyol@bingol.edu.tr}
\subjclass[2010]{53C15, 53C40.}
\keywords{The second fundamental form of a map, almost contact metric manifold, conformal submersion, slant submersion, conformal slant submersion, horizontal distribution.}

\begin{abstract}
Akyol M.A. [Conformal anti-invariant submersions from cosymplectic manifolds, Hacettepe
Journal of Mathematics and Statistic, 46(2), (2017), 177-192.] defined and studied conformal anti-invariant submersions from cosymplectic manifolds.
The aim of the present paper is to define and study the notion of conformal slant submersions (it means the Reeb vector field $\xi$ is a vertical vector field) from almost contact metric manifolds onto Riemannian manifolds as a generalization of Riemannian submersions, horizontally conformal submersions, slant submersions and conformal anti-invariant submersions. More precisely, we mention lots of examples and obtain the geometries of the leaves of $\ker\pi_{*}$ and $(\ker\pi_{*})^\perp,$ including the integrability of the distributions, the geometry of foliations, some conditions related to totally geodesicness and harmonicty of the submersions. Finally, we consider a decomposition theorem on total space of the new submersion.
\end{abstract}

\maketitle

\section{Introduction}
O'Neill \cite{O} and Gray \cite{Gray} independently studied the notion of Riemannian submersions between Riemannian manifolds in the 1960s. This notion is related with physics and have some applications in the Yang-Mills theory \cite{BL,W1}, Super-gravity and superstring theories \cite{IV1,M} and Kaluza-Klein theory \cite{BL1,IV}. For Riemannian submersions (See also: \cite{SMV,s2}). After that, by using the notion of Riemannian submersion and the condition of almost complex mapping, Watson \cite{watson} introduced almost hermitian submersions. He showed that
under the almost complex structure of the total manifold of the submersion the vertical distribution and the horizontal distribution invariant. 

As a generalization of this notion, \c{S}ahin \cite{s} defined the notion of anti-invariant Riemannian submersions from almost Hermitian manifolds. Afterwards, He also defined slant submersions from almost Hermitian manifolds in \cite{s1}. After that, many geometers study this area and obtain lots of results on the new topic.  (see: \cite{Ag}, \cite{As}, \cite{As1}, \cite{C}, \cite{GIP}, \cite{Gun1}, \cite{Gun2},  \cite{Park2},  \cite{Park4}, \cite{tastan}). Recent developments on the notion of Riemannian submersion could be found in the book, \cite{s3}.

In \cite{Chi4}, Chinea introduced the notion of almost contact Riemannian submersions between almost contact metric manifolds. He obtained the differential geometric properties among total space, fibers and base spaces.

A related topic of growing interest deals with the study of the so called \textit{horizontally conformal submersions}: these maps, which provide a natural generalization of Riemannian submersion, introduced independently Fuglede \cite{F} and Ishihara \cite{I}. As a generalization of holomorphic submersions, the notion of conformal holomorphic submersions  were defined by Gudmundsson and Wood \cite{G, GW} (see also: \cite{A}, \cite{A1}, \cite{As1}, \cite{BW}, \cite{Chi1}, \cite{Chi2}, \cite{Chi3}, \cite{G}, \cite{Ornea}). In 2017, Akyol and \c{S}ahin \cite{As2} defined a conformal slant submersion from an almost Hermitian manifolds onto a riemannian manifold. In this paper, we consider conformal slant submersions from a cosymplectic manifold onto a Riemannian  manifold.\\

The paper is organized as follows. In section 2, we  recall several notions and formulas for other sections.
In third section, we introduce conformal slant submersions from almost contact metric manifolds onto Riemannian manifolds,
mention a lot of examples and investigate the geometry of leaves of $\ker\pi_{*}$ (the vertical distribution) and $(\ker\pi_{*})^\perp$ (the horizontal distribution) and
find necessary and sufficient conditions for a conformal slant submersion to be totally
geodesic and harmonic, respectively. Finally, we consider a decomposition theorem on total space of the new submersion.

\section{Cosymplectic Manifolds}

A $(2n+1)$-dimensional $C^\infty$-manifold $N$ said to have an almost contact structure if there exist
on $N$ a tensor field $\phi$ of type (1,1), a vector field  $\xi$ and 1-form  $\eta$ satisfying:
\begin{equation}
\phi^{2}=-I+\eta\otimes\xi, \ \ \phi\xi=0,\ \ \eta o\phi=0,\ \ \eta(\xi)=1. \label{e.q:2.1}
\end{equation}
There always exists a Riemannian metric $g$ on an almost contact manifold $N$ satisfying the following
conditions
\begin{equation}
g_1(\phi X_1,\phi X_2)=g_1(X_1,X_2)-\eta(X_1)\eta(X_2),\ \ \ \eta(X_1)=g_1(X_1,\xi) \label{e.q:2.2}
\end{equation}
where $X_1,X_2\in\Gamma(TN).$ 

An almost contact structure $(\phi,\xi,\eta)$ is said to be normal if the almost complex structure
$J_1$ on the product manifold $N \times \mathbb{R}$ is given by
$$J_1(X_1,f\frac{d}{dt})=(\phi X_1-f\xi,\eta(X_1)\frac{d}{dt}),$$
where $f$ is a $C^\infty$-function on $N\times \mathbb{R}$ has no torsion i.e., $J_1$ is integrable. The condition
for normality in terms of $\phi,\xi$ and $\eta$ is $[\phi,\phi]+2d\eta\otimes\xi=0$ on $N,$ where
$[\phi,\phi]$ is the Nijenhuis tensor of $\phi.$ Finally, the fundamental two-form $\Phi$ is
defined $\Phi(X_1,\phi X_2)=g_1(X_1,\phi X_2).$

An almost contact metric structure $(\phi,\xi,\eta,g)$ is said to be cosymplectic, if it is normal and both
$\Phi$ and $\eta$ are closed (\cite{B}, \cite{L}), and the structure equation of a cosymplectic manifold is given by
\begin{equation}
(\nabla_{X_1}\phi)X_2=0 \label{e.q:2.3}
\end{equation}
for any $X_1,X_2$ tangent to $N,$ where $\nabla$ denotes the Riemannian connection of the metric $g$ on $N.$ Moreover,
for cosymplectic manifold
\begin{equation}
\nabla_{X_1}\xi=0. \label{e.q:2.4}
\end{equation}
The canonical example of cosymplectic manifold is given by the product $B^{2n}\times \mathbb{R}$ Kaehler manifold $B^{2n}(J,g)$
with the $\mathbb{R}$ real line. Now we will introduce a well known coymplectic manifold example of $\mathbb{R}^{2n+1}.$

\begin{example}(\cite{OL}, \cite{CE})\label{example1}
We consider $\mathbb{R}^{2n+1}$  with cartesian  coordinates
$(u_{i},v_{i},t)$ $(i=1,2,...,n)$  and its usual contact one-form $\eta=dt$.  The Reeb vector field $\xi$ is given by $\frac{\partial}{\partial
t}$ and its Riemannian metric $g_{\mathbb{R}^{2n+1}}$ and tensor field $\phi$ are given by
$$
g_{\mathbb{R}^{2n+1}}=(dt)^{2}+\sum_{i=1}^{n}((du_{i})^{2}+(dv_{i})^{2}),\,\,\,\varphi=
\left(%
\begin{array}{ccc}
  0 &  \delta_{ij} &0 \\
  -\delta_{ij} &0& 0 \\
 0 & 0 & 0 \\
\end{array}%
\right).
$$
This gives a cosymplectic manifold on $\mathbb{R}^{2n+1}.$  The vector fields
$e_i=\frac{\partial}{\partial v_i},$ $e_{n+i}=\frac{\partial}{\partial u_i},$ $\xi$ form a $\phi$-basis for the cosymplectic structure.
On the other hand, it can be shown that $\mathbb{R}^{2n+1}(\phi,\xi,\eta,g)$ is a cosymplectic manifold.
\end{example}

\section{Conformal submersions}
Let $\psi:(N,g_{N})\longrightarrow (B,g_{B})$ be a smooth map between Riemannian manifolds, and let $q\in N$. Then $\psi$ is
said to be \textit{horizontally weakly conformal or semi conformal} at $q$ \cite{BW} if either (i) $d\psi_{q}=0$, or
(ii) $d\psi_{q}$ is surjective and there exists a number $\Lambda(q)\neq0$ such that
$$
g_N(d\psi_{q}X,d\psi_{q}Y)=\Lambda(q)g_M(X,Y)\mbox{ }(X,Y\in\mathcal{H}_{q}).
$$
We call the point $q$ is of type (i) as a critical point if it satisfies the type (i), and we shall call the point $q$ a regular point if it satisfied the type (ii).
At a critical point, $d\psi_{q}$ has rank $0$; at a regular point, $d\psi_{q}$ has rank $n$
and $\psi$ is submersion. Further, the positive number $\Lambda(q)$ is called the \textit{square dilation} (of $\psi$ at $q$). The map $\psi$ is called \textit{horizontally weakly
conformal} or \textit{semi conformal} (on $N$) if it is horizontally weakly
conformal at every point of $N$ and it has no critical point, then we call it a (\textit{horizontally conformal submersion}).

A vector field $X_{1}\in\Gamma(TN)$ is called a basic vector field if $X_{1}\in\Gamma((ker\psi_{*})^\perp)$ and $\psi-$related with a vector field $\bar{X_{1}}\in\Gamma(TB)$
which means that $({\psi_{*}}_{q}X_{1q})=\bar{X_{1}}(\psi(q))\in\Gamma(TB)$ for any $q\in\Gamma(TN).$

Define O'Neill's tensors $T$ and $A$ by
\begin{equation}
A_{X_{1}}X_{2}=v\nabla_{hX_{1}}hX_{2} +v\nabla_{hX_{1}}vX_{2} \label{e.q:2.5}
\end{equation}
\begin{equation}
T_{X_{1}}X_{2}=h\nabla_{vX_{1}}vX_{2} +v\nabla_{vX_{1}}hX_{2} \label{e.q:2.6}
\end{equation}
where for any $X_{1},X_{2}\in\Gamma(TN)$ and $v,$ $h$ are the vertical and horizontal
projections (see \cite{FIP}). And also, by using (\ref{e.q:2.5}) and (\ref{e.q:2.6}), for $X_{1},X_{2}\in\Gamma((ker\psi_{*})^\perp)$ and $V_{1},V_{2}\in\Gamma(ker\psi_{*})$, we have
\begin{equation}
\nabla_{V_{1}}V_{2}=T_{V_{1}}V_{2}+\hat{\nabla}_{V_{1}}V_{2} \label{e.q:2.7}
\end{equation}
\begin{equation}
\nabla_{V_{1}}X_{1}=h\nabla_{V_{1}}X_{1}+T_{V_{1}}X_{1}\label{e.q:2.8}
\end{equation}
\begin{equation}
\nabla_{X_{1}}V_{1}=A_{X_{1}}V_{1} +v\nabla_{X_{1}}V_{1}\label{e.q:2.9}
\end{equation}
\begin{equation}
\nabla_{X_{1}}X_{2}=h\nabla_{X_{1}}X_{2}+A_{X_{1}}X_{2}\label{e.q:2.10}
\end{equation}
where $\hat{\nabla}_{V_{1}}V_{2}=v\nabla_{V_{1}}V_{2}$. If $X_{1}$ is basic, then $h\nabla_{V_{1}}X_{1}=A_{X_{1}}V_{1}$. Then we easily obtain
$-g_{N}(A_{X_{1}}E_{1},E_{2})=g_{N}(E_{1},A_{X_{1}}E_{2})$ and $-g_{N}(T_{V_{1}}E_{1},E_{2})=g_{N}(E_{1},T_{V_{1}}E_{2})$
for all $E_{1},E_{2}\in T_{x}N$. $T$ is exactly the second fundamental form of the fibres of $\psi$. For the special case when
$\psi$ is horizontally conformal we have the following:
\begin{proposition}
\textit{(\cite{G})} Let $\psi:(N,g_{N})\longrightarrow (B,g_{B})$ be a
horizontally conformal submersion with dilation $\lambda$ and $X_{1},X_{2}$ be
horizontal vectors, then
\begin{equation}
A_{X_{1}}X_{2}=\frac{1}{2}\{v[X_{1},X_{2}]-\lambda^{2}g_{N}(X_{1},X_{2})grad_{v}(\frac{1}{\lambda^{2}})\}.\label{e.q:2.11}
\end{equation}
\end{proposition}
Let $(N,g_{N})$ and $(B,g_{B})$ be Riemannian manifolds and suppose that $%
\psi:N \longrightarrow B$ is a smooth map between them. The second fundamental form of $\psi$ is given by
\begin{equation}
(\nabla\psi_{*})(X_{1},X_{2})=\nabla^{\psi}_{X_{1}}\psi_{*}(X_{2})-\psi_{*}(\nabla^{N}_{X_{1}}X_{2}) \label{e.q:2.12}
\end{equation}
for any $X_{1},X_{2}\in\Gamma(TN)$, where $\nabla^{\psi}$ is the pullback connection. it is obvious that the second fundamental form  ($\nabla\psi_*$) is symmetric.
\begin{lemma}\cite{U}
Let $(N,g_{N})$ and $(B,g_{B})$ be Riemannian manifolds and suppose that $%
\psi:N \longrightarrow B$ is a smooth map between them. Then we have
\begin{equation}
\nabla^{\psi}_{X_{1}}\psi_{*}(X_{2})-\nabla^{\psi}_{X_{2}}\psi_{*}(X_{1})-\psi_{*}([X_{1},X_{2}])=0 \label{e.q:2.13}
\end{equation}
for $X_{1},X_{2}\in\Gamma(TN)$.
\end{lemma}
From Lemma 2.1, for any $X_{1}$ a basic vector field and $V_{1}\in\Gamma(ker\psi_{*})$ 
we obtain $[X_{1},V_{1}]\in\Gamma(ker\psi_{*}).$
\begin{remark}
In this paper, we assume that all horizontal vector fields are basic vector fields.
\end{remark}
Recall that $\psi$ is called harmonic if the tension field $\tau(\psi)=trace(\nabla\psi_{*})=0$.  (for details, see \cite{BW}).
\begin{lemma} \cite{BW}
\label{lem1} Let $\psi:N \longrightarrow B$ be a horizontally conformal submersion. Then, we have
\begin{enumerate}
\item[(a)] $(\nabla \psi_*)(X_{1},X_{2})=X_{1}(\ln\lambda)\psi_* X_{2}+X_{2}(\ln\lambda)\psi_* X_{1}-g_{N}(X_{1},X_{2})\psi_*(\nabla\ln\lambda)$;
\item[(b)] $(\nabla \psi_*)(V_{1},V_{2})=-\psi_* (T_{V_{1}}V_{2})$;
\item[(c)] $(\nabla \psi_*)(X_{1},V_{1})=-\psi_*(\nabla^{N}_{X_{1}}V_{1})=-\psi_*(A_{X_{1}}V_{1})$
\end{enumerate}
for any $V_{1},V_{2}\in\Gamma(ker\psi_{*})$ and $X_{1},X_{2}\in(ker\psi_{*})^\perp$.
\end{lemma}

Finally, we will mention the following from \cite{P}. 

Let $g_N$ be a Riemannian metric tensor on the manifold $N=N_1\times N_2$
and assume that the canonical foliations $\mathcal{D}_{N_1}$ and $\mathcal{D}_{N_2}$ intersect
perpendiculary everywhere. Then $g_N$ is the metric tensor of a usual product of Riemannian manifolds $\Longleftrightarrow$ $\mathcal{D}_{N_1}$ and $\mathcal{D}_{N_2}$ are totally geodesic foliations.

\section{Conformal slant submersions}

In this section, we introduce the notion of conformal slant submersions from almost
contact metric manifolds onto Riemannian manifolds.  We mention lots of examples, obtain the integrability
of distributions,  the geometry of foliations, some conditions related to totally geodesicness and harmonicity of the map.

\begin{definition}
\label{def} Let $\psi: (N,\varphi,\xi,\eta,g_N)\longrightarrow (B,g_{B})$ be a horizontally conformal submersion, where 
$(N,\varphi,\xi,\eta,g_N)$ be an almost contact metric manifold and $(B,g_{B})$ be a Riemannian manifold. The map 
$\psi$ is said to be slant  if for any non zero vector $V_{1}\in\Gamma(ker\psi_{\ast})-<\xi>,$ the angle $\omega(V_{1})$ between $\varphi V_{1}$ and the space $ker\psi_{\ast}$ is a constant (which is independent of yhe choice of $p\in N$ and of $V_{1}\in\Gamma(ker\psi_{\ast})-<\xi>$).  The angle $\omega$ is called the slant angle of the conformal slant submersion. Conformal holomorphic submersion and conformal anti-invariant submersions are conformal slant submersions with $\omega=0$ and $\frac{\pi}{2}$, respectively. A conformal slant submersion which is not conformal holomorphic submersions nor conformal anti-invariant is called proper conformal slant submersion.
\end{definition}

Now, we present some examples.

\begin{example}
$\mathbb{R}^{5}$ has got a cosymplectic structure as in Example \ref{example1}. Let $\psi_1:\mathbb{R}^{5}\rightarrow \mathbb{R}^{2}$ be a map defined by $\psi_1(u_{1},u_{2},v_{1},v_{2},t)=e^{7}(u_{1}\cos\alpha-v_{1}\sin\alpha,u_{2}\sin\beta-v_{2}\cos\beta).$ Then, by direct calculations we obtain the Jacobian matrix of $\psi_1$ is:
\[
\psi_{1}*=e^{7}
\begin{bmatrix}
    \cos\alpha  & 0 & -\sin\alpha & 0 & 0 \\
   0      & \sin\beta & 0 & -\cos\beta & 0 \\
  \end{bmatrix}
.
\]
Since $rank(\psi_1)=2$, the map $\psi_1$ is a submersion. A straightforward computation, we see that
$$
ker\psi_{1}*=span\{V_{1}=\sin\alpha\frac{\partial}{\partial u_{1}}+\cos\alpha\frac{\partial}{\partial v_{1}},\,\,\,V_{2}=\cos\beta\frac{\partial}{\partial u_{2}}+\sin\beta\frac{\partial}{\partial v_{2}},\,\,\,V_{3}=\xi=\frac{\partial}{\partial t}\}
$$
and
$$
(ker\psi_{1}*)^{\bot}=span\{X_{1}=\cos\alpha\frac{\partial}{\partial u_{1}}-\sin\alpha\frac{\partial}{\partial v_{1}},\,\,\,
X_{2}=\sin\beta\frac{\partial}{\partial u_{2}}-\cos\beta\frac{\partial}{\partial v_{2}}\}.
$$
Then the map $\psi_1$ is a  conformal slant submersion with the slant angle $\omega$ and dilation $\lambda=e^{7}$ such that  $\cos\omega=\mid\cos(\alpha+\beta)\mid.$
\end{example}

\begin{example}
$\mathbb{R}^{5}$ has got a cosymplectic structure as in Example \ref{example1}. Let $\psi_2:\mathbb{R}^{5}\rightarrow \mathbb{R}^{2}$ be a map defined by $\psi_2(u_{1},u_{2},v_{1},v_{2},t)=\pi^{5}(\frac{u_1-u_2}{\sqrt{2}},v_2).$ Then, by direct calculations we obtain the Jacobian matrix of $\psi_2$ is:
\[
\psi_{2}*=\pi^{5}
\begin{bmatrix}
    \frac{1}{\sqrt{2}}  & -\frac{1}{\sqrt{2}} & 0 & 0 & 0 \\
   0      & 0 & 0 & 1 & 0 \\
  \end{bmatrix}
.
\]
Since the rank of this matris is equaled to 2, the map $\psi_2$ is a submersion. After some computation, we obtain
$$
ker\psi_{2}*=span\{H_{1}=\frac{\partial}{\partial u_{1}}+\frac{\partial}{\partial u_{2}},\,\,\,H_{2}=\frac{\partial}{\partial v_{1}},\,\,\,H_{3}=\xi=\frac{\partial}{\partial t}\}
$$
and
$$
(ker\psi_{2}*)^{\bot}=span\{Z_{1}=\frac{1}{\sqrt{2}}\big(\frac{\partial}{\partial u_{1}}-\frac{\partial}{\partial u_{2}}\big),\,\,\, Z_{2}=\frac{\partial}{\partial v_{2}}\}.
$$
Furthermore, $\psi_2(H_1)=-\frac{\partial}{\partial v_1}-\frac{\partial}{\partial v_2}$ and $\psi_2(H_{2})=\frac{\partial}{\partial u_{2}}$
imply that $\mid g(\psi_2(H_1),H_2)\mid=\frac{1}{\sqrt{2}}.$ So, the map $\psi_2$ is a  conformal slant submersion with the slant angle $\omega=\frac{\pi}{4}$ and dilation 
$\lambda=\pi^{5}.$ 
\end{example}

\begin{example}
$\mathbb{R}^{7}$ has got a cosymplectic structure as in Example \ref{example1}. Let $\psi_3:\mathbb{R}^{7}\rightarrow \mathbb{R}^{4}$ be a map defined by 
$\psi_3(u_{1},u_{2},u_3,v_{1},v_{2},v_3,t)=e^{11}(u_{1}, \frac{v_1-v_2}{\sqrt{2}},v_3,u_2).$ Then, by direct calculations we obtain the Jacobian matrix of $\psi_3$ is:
\[
\psi_{3}*=e^{11}
\begin{bmatrix}
    1  & 0 & 0 & 0 & 0 & 0 & 0 \\
   0   & 0 & 0 & \frac{1}{\sqrt{2}} &  -\frac{1}{\sqrt{2}} & 0 & 0\\
   0  & 0 & 0 & 0 & 0 & 1& 0 \\
    0   & 1 & 0  & 0 & 0 & 0 & 0\\
  \end{bmatrix}
.
\]
We easily see that the map $\psi_3$ is a submersion. After some computations, we derive
$$
ker\psi_{3}*=span\{\bar{H}_{1}=\frac{\partial}{\partial u_{3}},\,\,\,\bar{H}_{2}=\frac{1}{\sqrt{2}}(\frac{\partial}{\partial v_{1}}+\frac{\partial}{\partial v_{2}})
,\,\,\,\bar{H}_{3}=\xi=\frac{\partial}{\partial t}\}
$$
and
\begin{align*}
(ker\psi_{3}*)^{\bot}=span\{&\bar{Z}_{1}=\frac{\partial}{\partial u_{1}},\,\,\, \bar{Z}_{2}=\frac{1}{\sqrt{2}}(\frac{\partial}{\partial v_{1}}-\frac{\partial}{\partial v_{2}}),\,\,\,
\bar{Z}_{3}=\frac{\partial}{\partial u_{2}}, \bar{Z}_{4}=\frac{\partial}{\partial v_{3}}\}.
\end{align*}
Moreover, $\psi_3(\bar{H}_1)=-\frac{\partial}{\partial v_1}$ and $\psi_3(\bar{H}_{2})=\frac{1}{\sqrt{2}}(\frac{\partial}{\partial u_{2}}+\frac{\partial}{\partial u_{3}})$
imply that $\mid g(\psi_3\bar{H}_1,\bar{H}_2)\mid=\frac{1}{\sqrt{2}}.$ So, the map $\psi_3$ is a  conformal slant submersion with the slant angle $\omega=\frac{\pi}{4}$ and dilation 
$\lambda=e^{11}.$ 
\end{example}

\begin{example}
Let $(N,\varphi,\xi,\eta,g_{N})$ be an almost contact metric manifold. Suppose that $\sigma:TN\rightarrow N$ is the natural projection. Then, the map $\sigma$ is a conformal slant submersion with the slant angle $\omega=0$ and dilation $\lambda=1.$
\end{example}

\begin{example}
Let $(N,\varphi,\xi,\eta,g_{N})$ be an almost contact metric manifold and $(B,g_{B})$ a Riemannian manifold.  Suppose that $\psi_3:N\rightarrow B$ is a slant submersion \cite{erken}. Then, the map $\psi_3$ is a conformal slant submersion with  dilation $\lambda=1.$
\end{example}

\begin{example}
Let $(N^{2n+1},\varphi,\xi,\eta,g_{N})$ be an almost contact metric manifold and $(B^{2n},g_{B})$ a Riemannian manifold.  Suppose that $\psi_4:N\rightarrow B$ is a horizontally conformal  submersion with dilation $\lambda.$  Then, the map $\psi_4$ is a conformal slant submersion with the slant angle $\omega=\frac{\pi}{2}$  and  dilation $\lambda=1$ \cite{A}.
\end{example}

Let  $\psi:(N,\varphi,\xi,\eta,g_N)\longrightarrow (B,g_{B})$ be a conformal slant submersion from an almost contact metric  manifold
$(N,\varphi,\xi,\eta,g_N)$ to a Riemannian manifold $(B,g_{B})$. Then for any $V_{1}\in\Gamma(ker\psi_{\ast}),$ we write
\begin{equation}
\varphi V_{1}=DV_{1}+EV_{1}, \label{e.q:3.1}
\end{equation}
where $DV_{1}$ and $EV_{1}$  are vertical and horizontal components of $\varphi V_{1},$ respectively.\\
Given $X_{1}\in\Gamma(ker\psi_{\ast})^{\bot},$ we obtain
\begin{equation}
\varphi X_{1}=dX_{1}+eX_{1}, \label{e.q:3.2}
\end{equation}
where $dX_{1}\in\Gamma(ker\psi_{\ast})$ and $eX_{1}\in\Gamma(ker\psi_{\ast})^{\bot}.$\\
 We denote  the complementary  orthogonal distribution to decomposition $E(ker\psi_{*})$ in $(ker\psi_{\ast})^{\bot}$ by $\mu.$ Then we get
\begin{equation}
(ker\psi_{*})^\perp=E(ker\psi_{*})\oplus\mu. \label{e.q:3.3}
\end{equation}
From (\ref{e.q:2.4}), (\ref{e.q:2.7}) and (\ref{e.q:2.9}) we have
\begin{equation}
T_{V_{1}}\xi=0,\,\,\,\,A_{X_{1}}\xi=0 \label{e.q:3.4}
\end{equation}
for $X_{1}\in\Gamma((ker\psi_{*})^\perp)$ and $V_{1}\in\Gamma(ker\psi_{*}).$\\
By using (\ref{e.q:2.2}), (\ref{e.q:3.1}) and (\ref{e.q:3.2}) we get following result.
\begin{lemma}
Let $\psi$ be a conformal slant submersion from a cosymplectic manifold $(N,\varphi,\xi,\eta,g_{N})$ onto a Riemannian manifold $(B,g_{B})$. Then we obtain
\begin{equation}
DdX_{1}+deX_{1}=0, \label{e.q:3.5}
\end{equation}
\begin{equation}
EdX_{1}+e^{2}X_{1}=-X_{1},  \label{e.q:3.6}
\end{equation}
\begin{equation}
D^{2}V_{1}+dEV_{1}=\varphi^{2}V_{1}, \label{e.q:3.7}
\end{equation}
\begin{equation}
EDV_{1}+eEV_{1}=0 \label{e.q:3.8}
\end{equation}
for $X_{1}\in\Gamma((ker\psi_{*})^\perp)$ and $V_{1}\in\Gamma(ker\psi_{*})$.
\end{lemma}
Using (\ref{e.q:2.7}), (\ref{e.q:2.8}), (\ref{e.q:3.1}) and (\ref{e.q:3.2}) we obtain
\begin{equation}
 (\nabla_{V_{1}}D)V_{2}=dT_{V_{1}}V_{2}-T_{V_{1}}EV_{2}, \label{e.q:3.9}
 \end{equation}
 \begin{equation}
(\nabla_{V_{1}}E)V_{2} =eT_{V_{1}}V_{2}-T_{V_{1}}DV_{2} \label{e.q:3.10}
\end{equation}
here
\begin{equation}
(\nabla_{V_{1}}D)V_{2}=\hat{\nabla}_{V_{1}}DV_{2}-D\hat{\nabla}_{V_{1}}V_{2},
\label{eq:3.11}
\end{equation}
\begin{equation}
(\nabla_{V_{1}}E)V_{2}=h\nabla_{V_{1}}EV_{2}-E\hat{\nabla}_{V_{1}}V_{2},
\label{eq:3.12}
\end{equation}
for $V_{1},V_{2}\in\Gamma((ker\psi_{*}).$
We call $D$ and $E$ parallel if $\nabla D=0$ and $\nabla E=0,$ respectively.\\
Since the proof of the followig theorem is quite similar to Theorem 2.2 of \cite{cabrerizo}, thus we don't give the proof of it.
\begin{theorem}\label{teo1}
Let $\psi$ be a conformal slant submersion from a cosymplectic manifold $(N,\varphi,\xi,\eta,g_{N})$ onto a Riemannian manifold $(B,g_{B})$. Then, we obtain
\begin{equation}
D^{2}=-\cos^{2}\omega(I-\eta\otimes\xi) \label{e.q:3.13}
\end{equation}
for any  $V_{1}\in\Gamma(ker\psi_{*})$.
\end{theorem}
By (\ref{e.q:2.2}), (\ref{e.q:3.1}) and (\ref{e.q:3.13}) we have following result.
\begin{corollary}
Let $\psi$ be a conformal slant submersion from a cosymplectic manifold $(N,\varphi,\xi,\eta,g_{N})$ onto a Riemannian manifold $(B,g_{B})$. Then, we obtain
\begin{equation}
g_{N}(DV_{1},DV_{2})=\cos^{2}\omega(g_{N}(V_{1},V_{2})-\eta(V_{1})\eta(V_{2})),\label{e.q:3.14}
\end{equation}
\begin{equation}
 g_{N}(EV_{1},EV_{2})=\sin^{2}\omega(g_{N}(V_{1},V_{2})-\eta(V_{1})\eta(V_{2})),\label{e.q:3.15}
\end{equation}
for $V_{1},V_{2}\in\Gamma((ker\psi_{*}).$
\end{corollary}
\begin{proposition}
Let $\psi:(N,g_{N},\varphi,\eta,\xi)\rightarrow (B,g_{B})$ be a conformal  slant
submersion. If $N$ is a cosymplectic manifold and
$E$ is parallel with respect to $\nabla$ on $(ker\psi_{\ast}),$
then we have
\begin{equation}
T_{DV_{1}}DV_{1}=-\cos^{2}\omega T_{V_{1}}V_{1}\label{e.q:3.16}
\end{equation}
for any  $V_{1}\in\Gamma(ker\psi_{\ast}).$
\end{proposition}
\begin{proof}
 If $E$ is parallel, then
we derive $eT_{V_{1}}V_{2}=T_{V_{1}}DV_{2}$ for any $V_{1},V_{2}\in\Gamma(ker\psi_{\ast}).$
Interchanging the role of $V_{1}$ and $V_{2},$ we get $eT_{V_{2}}V_{1}=T_{V_{2}}DV_{1}.$
So, we have
$$
eT_{V_{1}}V_{2}-eT_{V_{2}}V_{1}=T_{V_{1}}DV_{2}-T_{V_{2}}DV_{1}
$$
Since $T$ is symmetric, we get $T_{V_{1}}DV_{2}=T_{V_{2}}DX_{1}.$ Then
substituting $V_{2}$ by $DV_{1}$ we get $T_{V_{1}}D^{2}V_{1}=T_{DV_{1}}DV_{1}.$ By (\ref{e.q:3.4}) and (\ref{e.q:3.13}) we obtain (\ref{e.q:3.16}).\\
\end{proof}
\begin{theorem}
Let $\psi$ be a conformal slant submersion from a cosymplectic manifold $(N^{2n+1},\varphi,\xi,\eta,g_{N})$ onto a Riemannian manifold $(B^{s},g_{B})$. Suppose that $E$ is parallel with slant angle $\omega\in[0,\frac{\pi}{2}).$ Then, all the fibres of the map $\psi$ are minimal.
\end{theorem}
\begin{proof}
Using (\ref{e.q:3.4}) and Lemma 5 of \cite{erken}, we have
$$
\tau(\psi)=\sum_{i=1}^{n-\frac{s}{2}}(\nabla\psi_{\ast})(E_{i},E_{i})=-\sum_{i=1}^{n-\frac{s}{2}}\psi_{\ast}(T_{E_{i}}E_{i}+T_{\sec\omega DE_{i}}\sec\omega DE_{i})-\psi_{\ast}(T_{\xi}\xi).
$$
Since $T_{\xi}\xi=0,$ we get
$$
\tau=-\sum_{i=1}^{n-\frac{s}{2}}\psi_{\ast}(T_{E_{i}}E_{i}+\sec^{2}\omega T_{DE_{i}}DE_{i}).
$$
By (\ref{e.q:3.16}), we have
$$
\tau=-\sum_{i=1}^{n-\frac{s}{2}}\psi_{\ast}(T_{E_{i}}E_{i}+\sec^{2}\omega(-\cos^{2}\omega T_{E_{i}}E_{i})=-\sum_{i=1}^{n-\frac{s}{2}}\psi_{\ast}(T_{E_{i}}E_{i}-T_{E_{i}}E_{i})=0.
$$
Thus, we prove that $\psi$ is harmonic.
\end{proof}

We now deal with the integrability of the distributions and the geometry of foliations.

\begin{theorem}
Let $\psi$ be a conformal slant submersion from a cosymplectic manifold $(N,\varphi,\xi,\eta,g_{N})$ onto a Riemannian manifold $(B,g_{B})$. Then the following conditions are equivalent to each other;
\begin{enumerate}
\item[(i)] The distribution $(ker\psi_{*})^\perp$ is integrable,
\item[(ii)]$\begin{aligned}[t]
&{\lambda^{-2}}g_{B}(\nabla_{X_{2}}^{\psi}\psi_{*}eX_{1}-\nabla^{\psi}_{X_{1}}\psi_{*}eX_{2},\psi_{*}EV_{1})\\
&=g_{N}(v\nabla_{X_{1}}dX_{2}+A_{X_{1}}eX_{2}-v\nabla_{X_{2}}dX_{1}-A_{X_{2}}eX_{1},DV_{1})\\
&+g_{N}(A_{X_{1}}dX_{2}-A_{X_{2}}dX_{1}-X_{_{1}}(\ln\lambda)eX_{2}+X_{2}(\ln\lambda)eX_{1}\\
&-eX_{2}(\ln\lambda)X_{1}+eX_{1}(\ln\lambda)X_{2}\\
&+2g_N(X_{1},eX_{2})(\nabla\ln\lambda),EV_{1})
\end{aligned}$
\end{enumerate}
for any $X_{1},X_{2}\in\Gamma((ker\psi_{*})^\perp)$ and $V_{1}\in\Gamma(ker\psi_{*})$.
\end{theorem}
\begin{proof}
In view of (\ref{e.q:2.2}), (\ref{e.q:2.3}) and (\ref{e.q:3.4}), we have
\begin{equation}
g_{N}([X_{1},X_{2}],V_{1})=g_{N}(\nabla_{X_{1}}\varphi X_{2}-\nabla_{X_{2}}\varphi X_{1},\varphi V_{1}). \label{e.q:3.17}
\end{equation}
for any $X_{1},X_{2}\in\Gamma((ker\psi_{*})^\perp)$ and $V_{1}\in\Gamma(ker\psi_{*})$. Then, using (\ref{e.q:3.1}), (\ref{e.q:3.2}) and (\ref{e.q:3.17}), we derive
\begin{align*}
g_{M}([X,Y],W)&=g_{N}(\nabla_{X_{1}}dX_{2},DV_{1})+g_{N}(\nabla_{X_{1}}dX_{2},EV_{1})\\
&+g_{N}(\nabla_{X_{1}}eX_{2},DV_{1})+g_{N}(\nabla_{X_{1}}eX_{2},EV_{1})\\ 
&-g_{N}(\nabla_{X_{2}}dX_{1},DV_{1})-g_{N}(\nabla_{X_{2}}dX_{1},EV_{1})\\
&-g_{N}(\nabla_{X_{2}}eX_{1},DV_{1})-g_{N}(\nabla_{X_{2}}eX_{1},EV_{1}).
\end{align*}
Using the property of $\psi,$ (\ref{e.q:2.9}) and (\ref{e.q:2.10}) we get
\begin{align*}
g_{N}([X_{1},X_{2}],V1)&=g_{N}(v\nabla_{X_{1}}dX_{2}+A_{X_{1}}eX_{2}-v\nabla_{X_{2}}dX_{1}-A_{X_{2}}eX_{1},DV_{1})\\
&+g_{N}(A_{X_{1}}dX_{2}+h\nabla_{X_{1}}eX_{2}-A_{X_{2}}dX_{1}-h\nabla_{X_{2}}eX_{1},EV_{1}) \\ &=g_{N}(v\nabla_{X_{1}}dX_{2}+A_{X_{1}}eX_{2}-v\nabla_{X_{2}}dX_{1}-A_{X_{2}}eX_{1},DV_{1})\\
&+g_{N}(A_{X_{1}}dX_{2}-A_{X_{2}}dX_{1},EV_{1}) \\
&+{\lambda^{-2}}g_{B}(\psi_{*}(h\nabla_{X_{1}}eX_{2}-\psi_{*}(h\nabla_{X_{2}}eX_{1},\psi_{*}EV_{1}).
\end{align*}
Thus, by (\ref{e.q:2.12}) and Lemma 3.2 we obtain
\begin{align*}
g_{N}([X_{1},X_{2}],V1)&=g_{N}(v\nabla_{X_{1}}dX_{2}+A_{X_{1}}eX_{2}-v\nabla_{X_{2}}dX_{1}-A_{X_{2}}eX_{1},DV_{1})\\
&+g_{N}(A_{X_{1}}dX_{2}-A_{X_{2}}dX_{1},EV_{1})\\
&+{\lambda^{-2}}g_{B}(-(\nabla\psi_{*})(X_{1},eX_{2})+\nabla_{X_{1}}^{\psi}\psi_{\ast}eX_{2}+(\nabla\psi_{*})(X_{2},eX_{1})-\nabla_{X_{2}}^{\psi}\psi_{\ast}eX_{1},\psi_{*}EV_{1})\\
&=g_{N}(v\nabla_{X_{1}}dX_{2}+A_{X_{1}}eX_{2}-v\nabla_{X_{2}}dX_{1}-A_{X_{2}}eX_{1},DV_{1})\\
&+g_{N}(A_{X_{1}}dX_{2}-A_{X_{2}}dX_{1},EV_{1})+{\lambda^{-2}}g_{B}(\nabla_{X_{1}}^{\psi}\psi_{\ast}eX_{2}-\nabla_{X_{2}}^{\psi}\psi_{\ast}eX_{1},\psi_{*}EV_{1})\\
&+{\lambda^{-2}}g_{B}(-X_{1}(\ln\lambda)\psi_{\ast}eX_{2}-eX_{2}(\ln\lambda)\psi_{\ast}X_{1}+g_{N}(X_{1},eX_{2})\psi_{\ast}(\nabla\ln\lambda)\\
&+X_{2}(\ln\lambda)\psi_{\ast}eX_{1}+eX_{1}(\ln\lambda)\psi_{\ast}X_{2}-g_{N}(X_{2},eX_{1})\psi_{\ast}(\nabla\ln\lambda),\psi_{\ast}EV_{1})\\
&=g_{N}(v\nabla_{X_{1}}dX_{2}+A_{X_{1}}eX_{2}-v\nabla_{X_{2}}dX_{1}-A_{X_{2}}eX_{1},DV_{1})\\
&+g_{N}(A_{X_{1}}dX_{2}-A_{X_{2}}dX_{1}-X_{1}(\ln\lambda)eX_{2}+X_{2}(\ln\lambda)eX_{1}-eX_{2}(\ln\lambda)X_{1}\\
&+eX_{1}(\ln\lambda)X_{2}+2g_{N}(X_{1},eX_{2})(\nabla\ln\lambda),EV_{1})\\
&+{\lambda^{-2}}g_{B}(\nabla_{X_{1}}^{\psi}\psi_{\ast}eX_{2}-\nabla_{X_{2}}^{\psi}\psi_{\ast}eX_{1},\psi_{*}EV_{1}).
\end{align*}
Hence, $(i)\Leftrightarrow (ii)$.
\end{proof}
 We  note that a horizontally conformal submersion $\psi:N\rightarrow B$ is said to be horizontally homothetic if the gradient of its dilation $\lambda$ is vertical, i.e.,  $h(grad\lambda)=0$ at $p\in N$, where $h$ is the projection on the horizontal space $(ker\psi_{*})^\perp.$

\begin{theorem}
Let $\psi$ be a conformal slant submersion from a cosymplectic manifold $(N,\varphi,\xi,\eta,g_{N})$ onto a Riemannian manifold $(B,g_{B})$. Suppose that the distribution $(ker\psi_{*})^\perp$ is integrable. Then the following conditions are equivalent to each other;
\begin{enumerate}
\item [(i)] The map $\psi$ is horizontally homothetic submersion.
\item [(ii)]$\begin{aligned}[t]
&{\lambda^{-2}}g_{B}(\nabla_{X_{2}}^{\psi}\psi_{*}eX_{1}-\nabla^{\psi}_{X_{1}}\psi_{*}eX_{2},\psi_{*}EV_{1})\\
&=g_{N}(v\nabla_{X_{1}}dX_{2}+A_{X_{1}}eX_{2}-v\nabla_{X_{2}}dX_{1}-A_{X_{2}}eX_{1},DV_{1})\\
&+g_{N}(A_{X_{1}}dX_{2}-A_{X_{2}}dX_{1},EV_{1})
\end{aligned}$
\end{enumerate}
for $X_{1},X_{2}\in\Gamma((ker\psi_{*})^\perp)$ and $V_{1}\in\Gamma(ker\psi_{*})$.
\end{theorem}
\begin{proof}
 For any $X_{1},X_{2}\in\Gamma((ker\psi_{*})^\perp)$ and $V_{1}\in\Gamma(ker\psi_{*}),$  by  hypothesis,  we get
\begin{eqnarray}
0=g_{N}([X_{1},X_{2}],V_{1})&=&g_{N}(v\nabla_{X_{1}}dX_{2}+A_{X_{1}}eX_{2}-v\nabla_{X_{2}}dX_{1}-A_{X_{2}}eX_{1},DV_{1})\nonumber\\
&+&g_{N}(A_{X_{1}}dX_{2}-A_{X_{2}}dX_{1}-X_{1}(\ln\lambda)eX_{2}+X_{2}(\ln\lambda)eX_{1}\\
&-&eX_{2}(\ln\lambda)X_{1}+eX_{1}(\ln\lambda)X_{2}+2g_{N}(X_{1},eX_{2})(\nabla\ln\lambda),EV_{1})\nonumber\\
&+&{\lambda^{-2}}g_{B}(\nabla_{X_{1}}^{\psi}\psi_{\ast}eX_{2}-\nabla_{X_{2}}^{\psi}\psi_{\ast}eX_{1},\psi_{*}EV_{1}).\label{e.q:3.18}
\end{eqnarray}
By (\ref{e.q:3.18}),  we obtain $(i)\Leftrightarrow (ii)$. Conversely, using  (\ref{e.q:3.18}), we have
\begin{eqnarray}
0&=&g_{N}(-X_{1}(\ln\lambda)eX_{2}+X_{2}(\ln\lambda)eX_{1}-eX_{2}(\ln\lambda)X_{1}+eX_{1}(\ln\lambda)X_{2}\nonumber\\
&+&2g_{N}(X_{1},eX_{2})(\nabla\ln\lambda),EV_{1}).\label{e.q:3.19}
\end{eqnarray}
If $X_{2}\in\Gamma(\mu),$ then by (\ref{e.q:3.3}) and (\ref{e.q:3.19}), we derive
\begin{eqnarray}
0&=&g_{N}(X_{2}(\ln\lambda)eX_{1}-\varphi X_{2}(\ln\lambda)X_{1}+2g_{N}(X_{1},\varphi X_{2})(\nabla\ln\lambda),EV_{1}).\label{e.q:3.20}
\end{eqnarray}
Now, taking $X_{1}=\varphi X_{2}$ in (\ref{e.q:3.20}), we obtain
\begin{eqnarray}
0&=&g_{N}(X_{2}(\ln\lambda)\varphi^{2}X_{2}-\varphi X_{2}(\ln\lambda)\varphi X_{2}+2g_{N}(\varphi X_{2},\varphi X_{2})(\nabla\ln\lambda),EV_{1})\nonumber\\
&=& 2g_{N}(X_{2},X_{2})g_{N}(\nabla\ln\lambda,EV_{1}), \label{e.q:3.21}
\end{eqnarray}
which implies
\begin{eqnarray}
g_{N}(\nabla\lambda,EV_{1})=0,\,\,\,\,V_{1}\in\Gamma(ker\psi_{\ast}).\label{e.q:3.22}
\end{eqnarray}
Taking $X_{1}=EV_{1}$ in (\ref{e.q:3.20}), we get
\begin{eqnarray}
0&=&g_{N}(X_{2}(\ln\lambda)eEV_{1}-\varphi X_{2}(\ln\lambda)EV_{1},EV_{1})\nonumber\\
&=& -\varphi X_{2}(\ln\lambda)g_{N}(EV_{1},EV_{1}),\nonumber\
\end{eqnarray}
which means
\begin{eqnarray}
g_{N}(\nabla\lambda,X_{3})=0,\,\,\,\,X_{3}\in\Gamma(\mu).\label{e.q:3.23}
\end{eqnarray}
Using (\ref{e.q:3.22}) and (\ref{e.q:3.23}),  we obtain $(ii)\Leftrightarrow(i).$
\end{proof}

\begin{theorem}
Let $\psi$ be a conformal slant submersion from a cosymplectic manifold $(N,\varphi,\xi,\eta,g_{N})$ onto a Riemannian manifold $(B,g_{B})$.  Then the following assertions are equivalent to each other;
\begin{enumerate}
\item [(i)] $(ker\psi_{*})^\perp$ defines a totally geodesic foliation on the total space.
\item [(ii)]$\begin{aligned}[t]
&\lambda^{-2}g_{B}(\nabla^{\psi}_{X_{1}}\psi_{\ast}X_{2},\psi_{\ast}EDV_{1})-\lambda^{-2}g_{B}(\nabla^{\psi}_{X_{1}}\psi_{\ast}eX_{2},\psi_{\ast}EV_{1})
=g_{N}(A_{X_{1}}dX_{2},EV_{1})\\
&+g_{N}(-X_{1}(\ln\lambda)eX_{2}-eX_{2}(\ln\lambda)X_{1}+g_{N}(X_{1},eX_{2})(\nabla\ln\lambda),EV_{1})\\
&-g_{N}(-X_{1}(\ln\lambda)X_{2}-X_{2}(\ln\lambda)X_{1}+g_{N}(X_{1},X_{2})(\nabla\ln\lambda),EDV_{1})
\end{aligned}$
\end{enumerate}
for $X_{1},X_{2}\in\Gamma((ker\psi_{*})^\perp)$ and $V_{1}\in\Gamma(ker\psi_{*}).$
\end{theorem}
\begin{proof}
 Given $X_{1},X_{2}\in\Gamma((ker\psi_{*})^\perp)$, $V_{1}\in\Gamma(ker\psi_{*}).$ and by  (\ref{e.q:2.2}), (\ref{e.q:2.3}) and (\ref{e.q:2.4}), we get
$$
g_{N}(\nabla_{X_{1}}X_{2},V_{1})=g_{N}(\nabla_{X_{1}}\varphi X_{2},\varphi V_{1}).
$$
From (\ref{e.q:2.3}), (\ref{e.q:2.9}), (\ref{e.q:2.10}), (\ref{e.q:3.1}), (\ref{e.q:3.2}) and (\ref{e.q:3.13}) we derive
\begin{align*}
g_{N}(\nabla_{X_{1}}X_{2},V_{1})&=g_{N}(\nabla_{X_{1}}\varphi X_{2},DV_{1}+EV_{1})\\
&=\cos^{2}\omega g_{N}(\nabla_{X_{1}}X_{2},V_{1})-g_{N}(\nabla_{X_{1}}X_{2},EDV_{1})\\
&+g_{N}(A_{X_{1}}dX_{2}+h\nabla_{X_{1}}eX_{2},EV_{1})\\
\sin^{2}\omega g_{N}(\nabla_{X_{1}}X_{2},V_{1})&=-g_{N}(\nabla_{X_{1}}X_{2},EDV_{1})+g_{N}(A_{X_{1}}dX_{2}+h\nabla_{X_{1}}eX_{2},EV_{1}).
\end{align*}
Using the property of $\psi,$ (\ref{e.q:2.12}) and Lemma 3.2, we have
\begin{eqnarray}
\sin^{2}\omega g_{N}(\nabla_{X_{1}}X_{2},V_{1})&=&-g_{N}(\nabla_{X_{1}}X_{2},EDV_{1})+g_{N}(A_{X_{1}}dX_{2},EV_{1})+\lambda^{-2}g_{B}(\psi_{\ast}h\nabla_{X_{1}}eX_{2},\psi_{\ast}EV_{1})\nonumber\\
&=&-g_{N}(\nabla_{X_{1}}X_{2},EDV_{1})+g_{N}(A_{X_{1}}dX_{2},EV_{1})\nonumber\\
&+&\lambda^{-2}g_{B}(-(\nabla\psi_{\ast})(X_{1},eX_{2})+\nabla^{\psi}_{X_{1}}\psi_{\ast}eX_{2},\psi_{\ast}EV_{1})\nonumber\\
&=&-g_{N}(\nabla_{X_{1}}X_{2},EDV_{1})+g_{N}(A_{X_{1}}dX_{2},EV_{1})+\lambda^{-2}g_{B}(\nabla^{\psi}_{X_{1}}\psi_{\ast}eX_{2},\psi_{\ast}EV_{1})\nonumber\\
&+&\lambda^{-2}g_{B}(-X_{1}(\ln\lambda)\psi_{\ast}eX_{2}-eX_{2}(\ln\lambda)\psi_{\ast}X_{1}+g_{N}(X_{1},eX_{2})\psi_{\ast}(\nabla\ln\lambda),\psi_{\ast}EV_{1})\nonumber\\
&=&-g_{N}(\nabla_{X_{1}}X_{2},EDV_{1})+g_{N}(A_{X_{1}}dX_{2},EV_{1})+\lambda^{-2}g_{B}(\nabla^{\psi}_{X_{1}}\psi_{\ast}eX_{2},\psi_{\ast}EV_{1})\nonumber\\
&+&g_{N}(-X_{1}(\ln\lambda)eX_{2}-eX_{2}(\ln\lambda)X_{1}+g_{N}(X_{1},eX_{2})(\nabla\ln\lambda),EV_{1}).\nonumber\
\end{eqnarray}
On the other hand,
\begin{eqnarray}
g_{N}(\nabla_{X_{1}}X_{2},EDV_{1})&=&\lambda^{-2}g_{B}(\psi_{\ast}\nabla_{X_{1}}X_{2},\psi_{\ast}EDV_{1})\nonumber\\
&=&\lambda^{-2}g_{B}(-(\nabla\psi_{\ast})(X_{1},X_{2})+\nabla^{\psi}_{X_{1}}\psi_{\ast}X_{2},\psi_{\ast}EDV_{1})\nonumber\\
&=&\lambda^{-2}g_{B}(-X_{1}(\ln\lambda)\psi_{\ast}X_{2}-X_{2}(\ln\lambda)\psi_{\ast}X_{1}+g_{N}(X_{1},X_{2})\psi_{\ast}(\nabla\ln\lambda)\nonumber\\
&+&\nabla^{\psi}_{X_{1}}\psi_{\ast}X_{2},\psi_{\ast}EDV_{1})\nonumber\\
&=&g_{N}(-X_{1}(\ln\lambda)X_{2}-X_{2}(\ln\lambda)X_{1}+g_{N}(X_{1},X_{2})(\nabla\ln\lambda),EDV_{1})\nonumber\\
&+&\lambda^{-2}g_{B}(\nabla^{\psi}_{X_{1}}\psi_{\ast}X_{2},\psi_{\ast}EDV_{1}).\nonumber\
\end{eqnarray}
Thus, we obtain $(i)\Leftrightarrow(ii)$.
\end{proof}

\begin{theorem}
Let $\psi$ be a conformal slant submersion from a cosymplectic manifold $(N,\varphi,\xi,\eta,g_{N})$ onto a Riemannian manifold $(B,g_{B})$. Suppose that the distribution $(ker\psi_{*})^\perp$ defines a totally geodesic foliation on the total space. Then the following assertions are equivalent to each other;
\begin{enumerate}
\item [(i)] The map $\psi$ is a horizontally homothetic submersion.
\item[(ii)] $\lambda^{-2}g_{B}(\nabla^{\psi}_{X_{1}}\psi_{\ast}X_{2},\psi_{\ast}EDV_{1})-\lambda^{-2}g_{B}(\nabla^{\psi}_{X_{1}}\psi_{\ast}eX_{2},\psi_{\ast}EV_{1})=g_{N}(A_{X_{1}}dX_{2},EV_{1})$
\end{enumerate}
for any $X_{2},X_{1}\in\Gamma((ker\psi_{*})^\perp)$ and $V_{1}\in\Gamma(ker\psi_{*})$.
\end{theorem}
\begin{proof}
Given  $X_{2},X_{1}\in\Gamma((ker\psi_{*})^\perp)$ and $V_{1}\in\Gamma(ker\psi_{*})$, from Theorem 4.5, we get
\begin{align}
&\lambda^{-2}g_{B}(\nabla^{\psi}_{X_{1}}\psi_{\ast}X_{2},\psi_{\ast}EDV_{1})-\lambda^{-2}g_{B}(\nabla^{\psi}_{X_{1}}\psi_{\ast}eX_{2},\psi_{\ast}EV_{1})\nonumber\\
&=g_{N}(A_{X_{1}}dX_{2},EV_{1})\nonumber\\
&+g_{N}(-X_{1}(\ln\lambda)eX_{2}-eX_{2}(\ln\lambda)X_{1}+g_{N}(X_{1},eX_{2})(\nabla\ln\lambda),EV_{1})\nonumber\\
&-g_{N}(-X_{1}(\ln\lambda)X_{2}-X_{2}(\ln\lambda)X_{1}+g_{N}(X_{1},X_{2})(\nabla\ln\lambda),EDV_{1}),\label{e.q:3.24}
\end{align}
which implies $(i)\Leftrightarrow(ii).$ Conversely, by (\ref{e.q:3.24}), we obtain
\begin{eqnarray}
0&=&g_{N}(-X_{1}(\ln\lambda)eX_{2}-eX_{2}(\ln\lambda)X_{1}+g_{N}(X_{1},eX_{2})(\nabla\ln\lambda),EV_{1})\nonumber\\
&-&g_{N}(-X_{1}(\ln\lambda)X_{2}-X_{2}(\ln\lambda)X_{1}+g_{N}(X_{1},X_{2})(\nabla\ln\lambda),EDV_{1}),\label{e.q:3.25}
\end{eqnarray}
Now, taking $X_{2}\in\Gamma(\mu),$ and  using (\ref{e.q:3.25}), we arrive at
\begin{eqnarray}
0&=&g_{N}(-\varphi X_{2}(\ln\lambda)X_{1}+g_{N}(X_{1},\varphi X_{2})(\nabla\ln\lambda),EV_{1})\nonumber\\
&-&g_{N}(-X_{2}(\ln\lambda)X_{1}+g_{N}(X_{1},X_{2})(\nabla\ln\lambda),EDV_{1}),\label{e.q:3.26}
\end{eqnarray}
 Taking $X_{1}=\varphi X_{2}$ in (\ref{e.q:3.26})  we find
 \begin{eqnarray}
0=g_{N}(\varphi X_{1},\varphi X_{2})(\nabla\ln\lambda),EV_{1}),\label{e.q:3.27}
\end{eqnarray}
 which implies
 \begin{eqnarray}
0=g_{N}(\nabla\lambda,EV_{1}),\,\,\,V_{1}\in\Gamma(ker\psi_{*}).\label{e.q:3.28}
\end{eqnarray}
Taking $X_{1}=EX_{2}$ in (\ref{e.q:3.26}) and by (\ref{e.q:3.3}), (\ref{e.q:3.15}) we have
\begin{eqnarray}
0&=&-g_{N}(EV_{1},EV_{1})g_{N}(\varphi X_{2},\nabla\ln\lambda)+g_{N}(X_{2},\nabla\ln\lambda)g_{N}(EV_{1},EDV_{1})\nonumber\\
&=&-g_{N}(EV_{1},EV_{1})g_{N}(\varphi X_{2},\nabla\ln\lambda),\label{e.q:3.29}
\end{eqnarray}
which implies
\begin{eqnarray}
0=g_{N}(\nabla\lambda,X_{3}),\,\,\,X_{3}\in\Gamma(\mu).\label{e.q:3.30}
\end{eqnarray}
Using (\ref{e.q:3.28}) and (\ref{e.q:3.30}), we obtain $(ii)\Leftrightarrow(i).$
\end{proof}
\begin{theorem}
Let $\psi$ be a conformal slant submersion from a cosymplectic manifold $(N,\varphi,\xi,\eta,g_{N})$ onto a Riemannian manifold $(B,g_{B})$.   Then for any $V_{1},V_{2}\in\Gamma(ker\psi_{*})$ and $X_{1}\in\Gamma((ker\psi_{*})^\perp)$ the following assertions are equivalent to each other;
\begin{enumerate}
\item [(i)] The distribution $ker\psi_{*}$ defines a totally geodesic foliation on the total space.
\item [(ii)]$\begin{aligned}[t]
g_{N}(\nabla_{V_{1}}EDV_{2},X_{1})=g_{N}(T_{V_{1}}EV_{2},dX_{1})+g_{N}(h\nabla_{V_{1}}EV_{2},eX_{1}).
\end{aligned}$
\end{enumerate}
\end{theorem}
\begin{proof}
For $V_{1},V_{2}\in\Gamma(ker\psi_{*})$ and $X_{1}\in\Gamma((ker\psi_{*})^\perp)$, by  (\ref{e.q:2.2}) and (\ref{e.q:2.3}) we obtain $g_{N}(\nabla_{V_{1}}V_{2},X_{1})=g_{N}(\nabla_{V_{1}}\varphi V_{2},\varphi X_{1}).$
Using (\ref{e.q:2.2}), (\ref{e.q:2.8}),(\ref{e.q:3.1}), (\ref{e.q:3.2})  and (\ref{e.q:3.13}) we arrive at
\begin{align*}
g_{N}(\nabla_{V_{1}}V_{2},X_{1})&=g_{N}(\nabla_{V_{1}}DV_{2}+EV_{2},\varphi X_{1})\\
&=-g_{N}(\nabla_{V_{1}}D^{2}V_{2}+\nabla_{V_{1}}EDV_{2},X_{1})+g_{N}(\nabla_{V_{1}}EV_{2},dX_{1}+eX_{1})\\
&=\cos^{2}\omega g_{N}(\nabla_{V_{1}}V_{2},X_{1})-g_{N}(\nabla_{V_{1}}EDV_{2},X_{1})\\
&+g_{N}(T_{V_{1}}EV_{2},dX_{1})+g_{N}(h\nabla_{V_{1}}EV_{2},eX_{1}).
\end{align*}
Thus, we have
\begin{align*}
\sin^{2}\omega g_{N}(\nabla_{V_{1}}V_{2},X_{1})&=-g_{N}(\nabla_{V_{1}}EDV_{2},X_{1})\\
&+g_{N}(T_{V_{1}}EV_{2},dX_{1})+g_{N}(h\nabla_{V_{1}}EV_{2},eX_{1})
\end{align*}
so that we obtain $(i)\Leftrightarrow(ii).$
\end{proof}
Now, we are going to investigate the harmonicity of $\psi.$

\cite{BW} Let $\psi$ be a horizontally conformal  submersion from a Riemannian  manifold $(N,g_{N})$ onto a Riemannian manifold $(B,g_{B})$ with dilation $\lambda$.
Then, the tension field $\tau(\psi)$ of $\psi$ is given by
\begin{eqnarray}
\tau(\psi)=-n\psi_{\ast}H+(2-s)\psi_{\ast}(\nabla\ln\lambda),\label{e.q:3.31}
\end{eqnarray}
where $H$ is the mean curvature vector field of the distribution $ker \psi_{\ast},$ $dimker \psi_{\ast}=n, $ $dimB=s.$\\
Using Theorem 4.2 and (\ref{e.q:3.31}), we have
\begin{corollary}
Let $\psi$ be a conformal slant submersion from a cosymplectic manifold $(N,\varphi,\xi,\eta,g_{N})$ onto a Riemannian manifold $(B,g_{B})$ with $dimB>2$.   Suppose that $E$ is parallel with the slant angle $\omega\in[0,\frac{\pi}{2})$. Then,  the following assertions are equivalent to each other;
\begin{enumerate}
\item [(i)] The map $\psi$ is harmonic.
\item [(ii)] The map $\psi$ is a horizontally homothetic submersion.
\end{enumerate}
\end{corollary}
\begin{corollary}
Let $\psi$ be a conformal slant submersion from a cosymplectic manifold $(N,\varphi,\xi,\eta,g_{N})$ onto a Riemannian manifold $(B,g_{B})$ with $dimB=2$.   Suppose that $E$ is parallel with the slant angle $\omega\in[0,\frac{\pi}{2})$. Then,  the map $\psi$ is harmonic.
\end{corollary}
Let $\psi$ be a conformal slant submersion from a cosymplectic manifold $(N,\varphi,\xi,\eta,g_{N})$ onto a Riemannian manifold $(B,g_{B})$. For $V_{1}\in\Gamma(ker\psi_{*})$ and $X_{1}\in\Gamma(\mu),$ we call the map $\psi (E ker\psi_{\ast},\mu)-$ totally geodesic if it satisfies $(\nabla\psi_{\ast})(EV_{1},X_{1})=0.$
\begin{theorem}
Let $\psi$ be a conformal slant submersion from a cosymplectic manifold $(N,\varphi,\xi,\eta,g_{N})$ onto a Riemannian manifold $(B,g_{B})$. Then,  the following assertions are equivalent to each other;
\begin{enumerate}
\item [(i)] The map $\psi$ is a horizontally homothetic submersion.
\item [(ii)] The map $\psi$ is $(E ker\psi_{\ast},\mu)-$ totally geodesic.
\end{enumerate}
\end{theorem}
\begin{proof}
For $V_{1}\in\Gamma(ker\psi_{*})$, $X_{1}\in\Gamma(\mu),$  from Lemma 3.2, we obtain
\begin{eqnarray}
(\nabla \psi_*)(EV_{1},X_{1})&=&EV_{1}(\ln\lambda)\psi_* X_{1}+X_{1}(\ln\lambda)\psi_* EV_{1}-g_{N}(EV_{1},X_{1})\psi_*(\nabla\ln\lambda)\nonumber\\
&=&EV_{1}(\ln\lambda)\psi_* X_{1}+X_{1}(\ln\lambda)\psi_* EV_{1}.\nonumber\
\end{eqnarray}
Since $g_{B}(\psi_* X_{1},\psi_* EV_{1})=\lambda^{2}g_{N}(X_{1},EV_{1})=0,$  $\{\psi_* X_{1},\psi_* EV_{1}\}$ is linearly independent for non-zero $V_{1},X_{1}.$ Thus we get $(i)\Leftrightarrow(ii).$
\end{proof}
\begin{theorem}
Let $\psi$ be a conformal slant submersion from a cosymplectic manifold $(N,\varphi,\xi,\eta,g_{N})$ onto a Riemannian manifold $(B,g_{B})$. Then,  the following assertions are equivalent to each other;
\begin{enumerate}
\item [(i)] the map $\psi$ is a totally geodesic map.
\item [(ii)] (a) $e(T_{V_{1}}DV_{2}+h\nabla_{V_{1}}EV_{2})+E(T_{V_{1}}EV_{2}+\hat{\nabla}_{V_{1}}DV_{2})=0,$ (b) $\psi$ is a horizontally homothetic map,
(c) $e(A_{X_{1}}DV_{1}+h\nabla_{X_{1}}EV_{1})+E(A_{X_{1}}EV_{1}+v\nabla_{X_{1}}DV_{1})=0$
\end{enumerate}
for $X_{1}\in\Gamma((ker\psi_{*})^\perp)$ and $V_{1},V_{2}\in\Gamma(ker\psi_{*})$.
\end{theorem}
\begin{proof}
Given $V_{1},V_{2}\in\Gamma(ker\psi_{*})$, using (\ref{e.q:2.2}), (\ref{e.q:2.3}) and (\ref{e.q:2.12}) we obtain $(\nabla \psi_{*})(V_{1},V_{2})=\psi_{*}(\varphi\nabla_{V_{1}}\varphi V_{2}).$
Using (\ref{e.q:2.7}), (\ref{e.q:2.8}), (\ref{e.q:3.1}) and (\ref{e.q:3.2}) we obtain
\begin{align*}
\nabla \psi_{*})(V_{1},V_{2})&=\psi_{*}(\varphi(\nabla_{V_{1}}DV_{2}+EV_{2})))\\
&=\psi_{*}(\varphi(T_{V_{1}}DV_{2}+\hat{\nabla}_{V_{1}}DV_{2}+T_{V1}EV_{2}+h\nabla_{V_{1}}EV_{2}))\\
&=\psi_{*}(dT_{V_{1}}DV_{2}+eT_{V_{1}}DV_{2}+D\hat{\nabla}_{V_{1}}DV_{2}+E\hat{\nabla}_{V_{1}}DV_{2}\\
&+DT_{V1}EV_{2}+ET_{V1}EV_{2}+d\nabla_{V_{1}}EV_{2}+e\nabla_{V_{1}}EV_{2})\\
&=\psi_{*}(eT_{V_{1}}DV_{2}+E\hat{\nabla}_{V_{1}}DV_{2}+ET_{V1}EV_{2}+e\nabla_{V_{1}}EV_{2}).
\end{align*}
Thus,  we have
\begin{eqnarray}
(\nabla \psi_{*})(V_{1},V_{2})=0\Leftrightarrow eT_{V_{1}}DV_{2}+E\hat{\nabla}_{V_{1}}DV_{2}+ET_{V1}EV_{2}+e\nabla_{V_{1}}EV_{2}=0.\label{e.q:3.32}
\end{eqnarray}
We claim that $\psi$ is a horizontally homothetic map if and only if $(\nabla \psi_{*})(X_{1},X_{2})=0$ for $X_{1},X_{2}\in\Gamma((ker\psi_{*})^\perp).$
From Lemma 3.2, we obtain
\begin{eqnarray}
(\nabla \psi_*)(X_{1},X_{2})=X_{1}(\ln\lambda)\psi_* X_{2}+X_{2}(\ln\lambda)\psi_* X_{1}-g_{N}(X_{1},X_{2})\psi_*(\nabla\ln\lambda) \label{e.q:3.33}
\end{eqnarray}
for $X_{1},X_{2}\in\Gamma((ker\psi_{*})^\perp)$ so that the part from left to right is obtained. Conversely, using (\ref{e.q:3.33}) we get
\begin{eqnarray}
0=X_{1}(\ln\lambda)\psi_* X_{2}+X_{2}(\ln\lambda)\psi_* X_{1}-g_{N}(X_{1},X_{2})\psi_*(\nabla\ln\lambda). \label{e.q:3.34}
\end{eqnarray}
Applying $X_{1}=X_{2}$ at (\ref{e.q:3.34}), we have
\begin{eqnarray}
0=2X_{1}(\ln\lambda)\psi_* X_{1}-g_{N}(X_{1},X_{1})\psi_*(\nabla\ln\lambda). \label{e.q:3.35}
\end{eqnarray}
Taking the inner product with $\psi_{\ast}X_{1}$ at (\ref{e.q:3.35}) we derive
$$
0=\lambda^{2}g_{N}(X_{1},X_{1})g_{N}(X_{1},\nabla\ln\lambda),
$$
which implies the result. For $X_{1}\in\Gamma((ker\psi_{*})^\perp)$ and $V_{1}\in\Gamma(ker\psi_{*})$,
using (\ref{e.q:2.2}), (\ref{e.q:2.3}) and (\ref{e.q:2.12}) we obtain $(\nabla \psi_{*})(X_{1},V_{1})=\psi_{*}(\varphi\nabla_{X_{1}}\varphi V_{1}).$\\
From (\ref{e.q:2.9}), (\ref{e.q:2.10}), (\ref{e.q:3.1}) and (\ref{e.q:3.2}) we get
\begin{align*}
\nabla \psi_{*})(X_{1},V_{1})&=\psi_{*}(\varphi(\nabla_{X_{1}}DV_{1}+EV_{1})))\\
&=\psi_{*}(\varphi(A_{X_{1}}DV_{1}+v\nabla_{X_{1}}DV_{1}+A_{X_{1}}EV_{1}+h\nabla_{X_{1}}EV_{1}))\\
&=\psi_{*}(eA_{X_{1}}DV_{1}+Ev\nabla_{X_{1}}DV_{1}+EA_{X_{1}}EV_{1}+e\nabla_{X_{1}}EV_{1}).
\end{align*}
From here,
\begin{eqnarray}
(\nabla \psi_{*})(X_{1},V_{1})=0\Leftrightarrow eA_{X_{1}}DV_{1}+Ev\nabla_{X_{1}}DV_{1}+EA_{X_{1}}EV_{1}+e\nabla_{X_{1}}EV_{1}=0.\label{e.q:3.36}
\end{eqnarray}
Thus, we have $(i)\Leftrightarrow (ii).$
\end{proof}
Finally, we consider a decomposition theorem. Denote by $N_{ker \psi_{\ast}}$ and $N_{(ker \psi_{\ast})^{\bot}}$ the integral manifolds of $ker \psi_{\ast}$ and $(ker \psi_{\ast})^{\bot}$, respectively. By Theorem 4.5 and Theorem 4.7, we have

\begin{theorem}
Let $\psi$ be a conformal slant submersion from a cosymplectic manifold $(N,\varphi,\xi,\eta,g_{N})$ onto a Riemannian manifold $(B,g_{B}).$  Then,  the following assertions are equivalent to each other;
\begin{enumerate}
\item [(i)] $(N,g_{N})$ is locally product manifold of the $N_{(ker \psi_{\ast})^{\bot}}\times N_{ker \psi_{\ast}}.$
\item [(ii)] $\begin{aligned}[t]
&\lambda^{-2}g_{B}(\nabla^{\psi}_{X_{1}}\psi_{\ast}X_{2},\psi_{\ast}EDV_{1})-\lambda^{-2}g_{B}(\nabla^{\psi}_{X_{1}}\psi_{\ast}eX_{2},\psi_{\ast}EV_{1})
=g_{N}(A_{X_{1}}dX_{2},EV_{1})\\
&+g_{N}(-X_{1}(\ln\lambda)eX_{2}-eX_{2}(\ln\lambda)X_{1}+g_{N}(X_{1},eX_{2})(\nabla\ln\lambda),EV_{1})\\
&-g_{N}(-X_{1}(\ln\lambda)X_{2}-X_{2}(\ln\lambda)X_{1}+g_{N}(X_{1},X_{2})(\nabla\ln\lambda),EDV_{1}),\\
&g_{N}(\nabla_{V_{1}}EDV_{2},X_{1})=g_{N}(T_{V_{1}}EV_{2},dX_{1})+g_{N}(h\nabla_{V_{1}}EV_{2},eX_{1})
\end{aligned}$
\end{enumerate}
for $X_{1},X_{2}\in\Gamma((ker\pi_{*})^\perp)$ and $V_{1},V_{2}\in\Gamma(ker\pi_{*})$.
\end{theorem}

\end{document}